\documentclass{article}
\usepackage{amssymb}
\usepackage{latexsym}
\usepackage[latin1]{inputenc}

\usepackage{graphicx}
\usepackage{amsthm}
\usepackage{epsfig}

\input xy
\xyoption{all}
\newdir{ >}{{}*!/-12pt/@{>}}

\title{Whitehead and Ganea constructions for fibrewise sectional category}
\author{J.M. Garc\'{\i}a-Calcines * \\[0.3pc]
         }

\date{}

\begin{document}

\maketitle

\footnote{*email address: jmgarcal@ull.es \hfill\break This work
has been supported by Ministerior de Educaci\'on y Ciencia grant
MTM2009-12081 and FEDER. \hfill\break 2000 {\em Mathematics
Subject Classification}: 55R70, 55P90, 55U35, 55M30.
\newline {\em Keywords and phrases}: Fibrewise space, fibrewise homotopy, J-category,
fibrewise sectional category, fibrewise Lusternik-Schnirelmann
category, topological complexity.\hfill\break }

\setlength{\baselineskip}{1.2pc}         %longitud entre l\'{\i}neas.
\setlength{\parskip} {0.3pc}         %longitud entre p\'{a}rrafos.
\newtheorem{theorem}{Theorem}[section]

\newtheorem{lemma}[theorem]{Lemma}
\newtheorem{definition}[theorem]{Definition}
\newtheorem{proposition}[theorem]{Proposition}
\newtheorem{corollary}[theorem]{Corollary}

\newtheorem{remark}[theorem]{Remark}
\newcommand{\Dt}{{\sf Proof:}}
\newcommand{\cqd}{\vspace{-15pt} \hfill$\blacksquare$}

\begin{abstract}
We introduce the notion of fibrewise sectional category via a
Whitehead-Ganea construction. Fibrewise sectional category is the
analogue of the ordinary sectional category in the fibrewise
setting and also the natural generalization of the fibrewise
unpointed LS category in the sense of Iwase-Sakai. On the other
hand the fibrewise pointed version is the generalization of the
fibrewise pointed LS category in the sense of James-Morris. After
giving their main properties we also establish some comparisons
between such two versions.
\end{abstract}

\section*{Introduction}
The sectional category $\mbox{secat}(p)$ of any map
$p:E\rightarrow B$ is the least non negative integer $k$ such that
there exists a cover of $B$ constituted by $k+1$ open subsets on
each of which $p$ has a local homotopy section. When $p$ is a
fibration, then we can consider local strict sections in the
definition, retrieving the usual notion of sectional category (or
Schwarz genus) of $p$ (see \cite{Sch}). This is a lower bound of
the Lusternik-Schnirelmann category of the base space and also a
generalization as $\mbox{secat}(p)=\mbox{cat}(B)$ when $E$ is
contractible. Apart from its usefulness in mathematical problems
such as the computation of the roots of a complex polynomial, the
embedding problem or the classification of bundles, the sectional
category is also crucial for the notion of topological complexity
of a space. The topological complexity of a space $X$, denoted as
$\mbox{TC}(X),$ is the sectional category of the evaluation
fibration $\pi :X^I\rightarrow X\times X,$ $\alpha \mapsto (\alpha
(0),\alpha (1)).$ It was established by Farber \cite{F,F2} in
order to face the motion planning problem in robotics from a
topological perspective. The topological complexity is also
interesting in algebraic topology itself, independently of its
original purpose in robotics, as it is closely related to
difficult tasks such as the immersion problem for real projective
spaces \cite{F-T-Y}. Since its apparition, this relatively new
numerical homotopy invariant has been of great interest for many
researchers working on applied algebraic topology. It is
remarkable the work of Iwase-Sakai \cite{I-S} who related the
topological complexity to what they call fibrewise unpointed LS
category $\mbox{cat}^*_B(-)$ in the fibrewise setting. Namely, if
$\Delta _X$ denotes the diagonal map, $pr_2$ the projection onto
the second factor, and $d(X)\equiv X\stackrel{\Delta _X
}{\longrightarrow }{X\times X}\stackrel{pr_2}{\longrightarrow }X$
the induced fibrewise pointed space over $X$, then they proved
that $\mbox{TC}(X)=\mbox{cat}^*_X(d(X)).$ They also work with a
pointed version in the fibrewise context and consider
$\mbox{TC}^M(X)$ the monoidal topological complexity and
$\mbox{cat}^B_B(-)$ the fibrewise pointed LS category in the sense
of James-Morris \cite{J-M}. In this case the equality
$\mbox{TC}^M(X)=\mbox{cat}^X_X(d(X))$ holds.

The notion of fibrewise unpointed LS category of Iwase-Sakai, and
the notion of fibrewise poin\-ted LS category of James-Morris
suggest a natural generalization, an analogue notion of sectional
category in the category of fibrewise (pointed) spaces over a
space $B.$  The aim of this paper is to establish such a
generalization, the \emph{fibrewise sectional category}, denoted
by $\mbox{secat}_B(-),$ as well as a Whitehead-Ganea approach of
it. We also present its pointed version, $\mbox{secat}_B^B(-),$
and some of their most important properties. In order to present
our work we have divided the paper into four sections. In the
first section we establish a Str{\o}m-type model category and give
some background about fibrewise homotopy such as fibrewise
homotopy pullbacks, pushouts and joins. In the second section we
introduce the main notion of the paper, the one of fibrewise
sectional category, and give its main properties. Among them its
Whitehead-Ganea approach (so it can be considered as an abstract
sectional category in the sense of \cite{K} or \cite{D-GC-R-M-R})
in which we will strongly use the results about fibrewise joins
given in the previous section. This can be summarized in the
following theorem. For details about its statement the reader is
referred to the first sections of the paper:
\begin{theorem}
Let $f:E\rightarrow X$ be any fibrewise map between normal spaces,
or a closed fibrewise cofibration with $X$ normal. Then the
following statements are equivalent:
\begin{enumerate}
\item[(i)] $\mbox{secat}_B(f)\leq n$

\item[(ii)] The diagonal map $\Delta _{n+1}:X\rightarrow \prod
_B^{n+1}X$ factors, up to fibrewise homotopy, through the
fibrewise sectional n-fat wedge
$$\xymatrix{ {X} \ar[r] \ar[dr]_{\Delta _{n+1}} & {T^n_B(f)} \ar[d]^{j_n} \\ & {\prod _B^{n+1}X} }$$

\item[(iii)] The $n$-th fibrewise Ganea map
$p_n:G^n_B(f)\rightarrow X$ admits a fibrewise homotopy section.
\end{enumerate}
\end{theorem}
It is also important the relationship of fibrewise sectional
category with the ordinary sectional category. In this sense we
have obtained the following result, which is not true in general.
Here by a fibrant space over $B$ we mean a fibrewise space over
$B$ in which the projection is a Hurewicz fibration.
\begin{theorem}
Let $f:E\rightarrow X$ be a fibrewise map between fibrant spaces
over $B.$ Then $\mbox{secat}_B(f)=\mbox{secat}(f).$
\end{theorem}
The third section is dedicated to the fibrewise pointed case,
where similar results are displayed and proved. Finally, in the
last section we compare the unpointed and the pointed versions. As
we will see, they are not so different as one might think. We
present two results, being quite general the first of them.
\begin{theorem}
Let $f:E\rightarrow X$ be a fibrewise pointed map in
$\mathbf{Top}_w(B)$ between normal spaces, or a closed fibrewise
cofibration with $X$ normal. Then
$$\mbox{secat}_B(f)\leq \mbox{secat}_B^B(f)\leq \mbox{secat}_B(f)+1$$
\end{theorem}

The second one, which closes our paper, is more restrictive but
also interesting. It is a generalization of a result given by
Dranishnikov when he compares topological complexity and monoidal
topological complexity.
\begin{theorem}
Let $f:E\rightarrow X$ a fibrewise pointed map between pointed
fibrant and cofibrant spaces over $B.$ Suppose that $B$ is a
CW-complex and $X$ a paracompact Hausdorff space satisfying the
following conditions:
\begin{enumerate}
\item[(i)] $f:E\rightarrow X$ is a $k$-equivalence ($k\geq 0$);

\item[(ii)] $\mbox{dim}(B)<(\mbox{secat}_B(f)+1)(k+1)-1.$
\end{enumerate}
\noindent Then $\mbox{secat}_B(f)=\mbox{secat}_B^B(f).$
\end{theorem}

Of course, in all these theorems we also obtain corollaries
replacing in the corresponding statements fibrewise sectional
category by fibrewise unpointed LS category or fibrewise pointed
sectional category by fibrewise pointed LS category. Therefore, by
Iwase-Sakai, we also retrieve known results about topological
complexity and monoidal topological complexity.

\section{Fibrewise homotopy theory.}

We begin by giving some preliminary definitions and results on
fibrewise homotopy theory that will be important throughout all
the paper. For basic notation and terminology theory we have
considered the reference \cite{C-J}.

Let $B$ be a fixed topological space. A \emph{fibrewise space}
over $B$ consists of a pair $(X,p_X),$ where $X$ is a topological
space and $p_X:X\rightarrow B$ a map from $X$ to $B.$ The map
$p_X$ is usually called the \emph{projection}. If there is not
ambiguity we will denote $X$ the fibrewise space $(X,p_X).$ If $X$
and $Y$ are fibrewise spaces, then a fibrewise map (over $B$)
$f:X\rightarrow Y$ is just a map $f:X\rightarrow Y$ satisfying
$p_Yf=p_X$
$$\xymatrix{
{X} \ar[rr]^f \ar[dr]_{p_X} & & {Y} \ar[dl]^{p_Y} \\
 & {B} & }$$

The corresponding category of fibrewise spaces and fibrewise maps
over $B$ will be denoted as $\mathbf{Top}_B.$ %Regarding $B$ as a
%fibrewise space over itself using the identity as the projection,
%we note that $B$ is the final object in $\mathbf{Top}_B.$ The
%initial object is $\emptyset .$

If $X$ and $Y$ are fibrewise spaces, then the binary fibrewise
product of $X$ and $Y$ will be denoted by $X\times _BY=\{(x,y)\in
X\times Y:p_X(x)=p_Y(y)\}.$ We will denote by $\prod ^n_BX$ the
product of n copies of a given fibrewise space $X.$

The \emph{fibrewise cylinder} of a fibrewise space $X$ is just the
usual product space $X\times I$ (where $I$ denotes the closed unit
interval $[0,1]$) together with the composite ${X\times
I}\stackrel{pr}{\longrightarrow }X\stackrel{p_X}{\longrightarrow
}B$ as the projection $p_{X\times I}.$ We will denote by $I_B(X)$
the fibrewise cylinder of $X$. The definition of \emph{fibrewise
homotopy} $\simeq _B$ between fibrewise maps comes naturally as
well as the notion of \emph{fibrewise homotopy equivalence}.

\subsection{Axiomatic homotopy for fibrewise spaces}

Now, for any fibrewise space $X$ take the following pullback in
the category $\mathbf{Top}$ of spaces and maps. Here $X^I$ (and
$B^I$) denotes the free path-space provided with the cocompact
topology, $p_X^I$ is the obvious map induced by $p_X$ and
$c:B\rightarrow B^I$ is the map that carries any $b\in B$ to the
natural constant path $c_b$ in $B^I$.
$$\xymatrix{
{P_B(X)} \ar[r] \ar[d] & {X^I} \ar[d]^{p_X^I} \\
{B} \ar[r]_c & {B^I} }$$ Thus, $P_B(X)=B\times
_{B^I}X^I=\{(b,\alpha )\in B\times X^I:c_b=p_X\alpha \}$ being
$P_B(X)\rightarrow B$ the obvious projection onto $B.$ $P_B(X)$ is
called the \emph{fibrewise cocylinder} of $X$ (or \emph{fibrewise
free path space} of $X$). The fibrewise cylinder and fibrewise
cocylinder constructions give rise to functors
$$I_B,P_B:\mathbf{Top}_B\rightarrow \mathbf{Top}_B$$
Associated to the functor $I_B$ there are defined natural
transformations $i_0,i_1:X\rightarrow I_B(X)$ and $\rho
:I_B(X)\rightarrow X$ (given by $i_{\varepsilon
}(x)=(x,\varepsilon)$ and $\rho (x,t)=x$). Analogously, associated
to $P_B$ there are defined natural transformations
$d_0,d_1:P_B(X)\rightarrow X$ and $c:X\rightarrow P_B(X)$ (given
by $d_{\varepsilon }(b,\alpha )=\alpha (\varepsilon )$ and
$c(x)=(p_X(x),c_x)$). Moreover, we have that $(I_B,P_B)$ is an
\emph{adjoint pair} in the sense of Baues (see \cite[p.29]{B}).

A fibrewise map$j:A\rightarrow X$ is said to be a \emph{fibrewise
cofibration} (\emph{over $B$}) if it satisfies the Homotopy
Extension Property, that is, for any fibrewise map $f:X\rightarrow
Y$ and any fibrewise homotopy $H:I_B(A)\rightarrow Y$ such that
$Hi_0=fj,$ there exists a fibrewise homotopy
$\widetilde{H}:I_B(X)\rightarrow Y$ such that $\widetilde{H}i_0=f$
and $\widetilde{H}I_B(j)=H$
$$\xymatrix{
{A} \ar[d]_j \ar[r]^{i_0} & {I_B(A)} \ar[d]^H
\ar@/^2pc/[ddr]^{I_B(j)} & \\ {X} \ar[r]^f \ar@/_2pc/[drr]_{i_0} &
{Y} &
\\ & & {I_B(X)} \ar@{.>}_{\widetilde{H}}[ul]}$$

As known, the fibrewise cofibrations are cofibrations in the usual
sense. One has just to take into account that, for a given space
$Z,$ the product $B\times Z$ is a fibrewise space considering the
canonical projection $B\times Z\rightarrow B.$ Therefore,
$j:A\rightarrow X$ is a \emph{fibrewise embedding}, that is,
$j:A\rightarrow j(A)$ is a fibrewise homeomorphism. Then we can
consider, without loss of generality, that the fibrewise
cofibrations are pairs of the form $(X,A).$ Such pairs are also
called \emph{fibrewise cofibred pairs}.

An important characterization of fibrewise cofibred pair, slightly
changed in \cite[Prop 5.2.4]{M-S} and also proved in \cite{C-J} in
the closed case, is given by what is called a \emph{fibrewise
Str{\o}m structure}.

\begin{proposition}\label{strom}
Let $(X,A)$ be a fibrewise pair. Then $(X,A)$ is fibrewise
cofibred if and only if $(X,A)$ admits a fibrewise Str{\o}m
structure, that is, a pair $(\varphi ,H)$ consisting of:
\begin{enumerate}

\item[(i)] A map $\varphi :X\rightarrow I$ satisfying $A\subseteq \varphi
^{-1}(\{0\});$

\item[(ii)] A fibrewise homotopy $H:I_B(X)\rightarrow X$
satisfying $H(x,0)=x,$ $H(a,t)=a$ for all $x\in X,$ $a\in A,$
$t\in I$, and $H(x,t)\in A$ whenever $t>\varphi (x).$
\end{enumerate}
If $A$ is closed the $\varphi $ can be taken so that $A=\varphi
^{-1}(\{0\}).$
\end{proposition}

An interesting consequence of Proposition \ref{strom} is the
following fact. Recall that given $A,U$ subspaces of a topological
space $X,$ $U$ is said to be a \emph{halo} of $A$ in $X$ if there
exists a map $\varphi  :X\rightarrow I$ such that $A\subseteq
\varphi ^{-1}(\{0\})$ and $\varphi ^{-1}([0,1))\subseteq U.$

\begin{corollary}\label{rem2}
Given $(X,A)$ any fibrewise cofibred pair, then $A$ is a fibrewise
strong deformation retract of an open neighborhood $U$ in $X.$
Such open subset $U$ is a halo of $A.$
\end{corollary}

 We will be particularly interested in \emph{closed}
fibrewise cofibred pairs (or closed fibrewise cofibrations), which
are closed fibrewise pairs $(X,A)$ (i.e., $A$ is a closed subspace
of $X$) such that $(X,A)$ is fibrewise cofibred. As in the
classical topological case, this is not a very restrictive
condition; for instance, if $(X,A)$ is any fibrewise cofibred pair
with $X$ Hausdorff, then necessarily $A$ is a closed subspace of
$X.$

 A \emph{fibrewise fibration} is a fibrewise map
$p:E\rightarrow Y$ such that it verifies Homotopy Lifting Property
with respect to any fibrewise space
$$\xymatrix{
{Z} \ar[r]^f \ar[d]_{i_0} & {E} \ar[d]^p \\
{I_B(Z)} \ar[r]_H \ar@{.>}[ur] & {Y} }$$

If $p:E\rightarrow Y$ is any fibrewise map such that it is a
Hurewicz fibration, then $p$ is a fibrewise fibration. However, in
general, the converse is not true. For instance, if $X$ is a
fibrewise space, then $p_X:X\rightarrow B$ is always a fibrewise
fibration, but $p_X$ need not be a Hurewicz fibration.

Denote by $fib_B,$ $\overline{cof}_B$ and $he_B$ the classes of
fibrewise fibrations, closed fibrewise cofibrations (equivalently,
closed fibrewise cofibred pairs) and fibrewise homotopy
equivalences, respectively. It is not hard to check that
$\mathbf{Top}_B$ is an $I$-category and a $P$-category in the
sense of Baues (see \cite[p.31]{B} for definitions). More is true,
one can also check that the Relative Homotopy Lifting Property
holds, i.e. if $(X,A)$ is any closed fibrewise cofibred pair and
$p:E\rightarrow Y$ any fibrewise fibration, then any commutative
diagram in $\mathbf{Top}_B$ of the form
$$\xymatrix{
{X\times \{0\}\cup A\times I} \ar[rr] \ar@{^(->}[d] & & {E} \ar@{>>}[d]^p \\
{I_B(X)} \ar[rr] & & {Y}}$$ \noindent admits a lift. All these
facts are summarized in the following proposition
\begin{proposition}\label{prev}
The category $\mathbf{Top}_B$ together with the classes of
$\overline{cof}_B,$ $fib_B$ and $he_B$ has an IP category
structure in the sense of Baues. In particular, $\mathbf{Top}_B$
is a cofibration category and a fibration category of Baues.
\end{proposition}

Following the reasonings given by Str{\o}m in \cite{S} we also
have the following theorem, which should be compared with \cite[Th
5.2.8]{M-S}.

\begin{theorem}\label{main}
The category $\mathbf{Top}_B$ together with the classes of
$\overline{cof}_B,$ $fib_B$ and $he_B$ has a proper closed model
category structure in the sense of Quillen.
\end{theorem}

\subsection{Homotopy pushouts, pullbacks and joins in the fibrewise
setting.}

In this subsection we will establish the Cube Lemma in
$\mathbf{Top}_B.$ This result will be crucial in order to deal
with a Whitehead-Ganea type characterization for fibrewise
sectional category. First we need to present the notions of
\emph{fibrewise homotopy pullback} and \emph{fibrewise homotopy
pushout}, which are simply the corresponding homotopy limis in the
fibrewise axiomatic setting. Given a fibrewise homotopy
commutative diagram
$$\xymatrix@R=0.5cm@C=0.5cm{
  X \ar[dd]_{g} \ar[rr]^{f} & & Y \ar[dd]^{h} \\
  & & \ar@/_.8pc/[ld]_{H} \\
  Z \ar[rr]_{k} & & K   } $$

\vspace{-2.1cm}
\begin{flushright}
$(*)$
\end{flushright}
\vspace{1cm}

\noindent with a fibrewise homotopy $H:hf\simeq _B kg,$ there is
another homotopy commutative diagram
$$\xymatrix@R=0.5cm@C=0.5cm{
  E_{h,k} \ar[dd]_{q} \ar[rr]^{p} & & Y \ar[dd]^{h} \\
  & & \ar@/_.8pc/[ld]_{G} \\
  Z \ar[rr]_{k} & & K   }$$

\vspace{-2.1cm}
\begin{flushright}
$(**)$
\end{flushright}
\vspace{1cm}

\noindent in which $E_{h,k}=\{(y,(b,\theta ),z)\in Y\times
P_B(K)\times Z\hspace{3pt};\hspace{3pt}h(y)=\theta
(0)\hspace{3pt},\hspace{3pt}k(z)=\theta (1)\}$ with the natural
projection $E_{h,k}\rightarrow B,$ given by $(y,(b,\theta
),z)\mapsto b.$ Here $p$ and $q$ are the obvious restrictions of
the projections and $G$ is the fibrewise homotopy defined as
$G(y,(b,\theta ),z,t)=\theta (t).$ There is a \emph{fibrewise
whisker map } $w:X\rightarrow E_{h,k}$ given by
$w(x)=(f(x),(p_X(x),H(x,-)),g(x))$ satisfying $pw=f,$ $qw=g$ and
$G(w\times id)=H.$ The homotopy commutative square (*) is said to
be a fibrewise homotopy pullback whenever $w$ is a fibrewise
homotopy equivalence. Given $h$ and $k$ fibrewise maps there
always exist their fibrewise homotopy pullback. The square (**) is
called the \emph{standard fibrewise homotopy pullback}.

There is the dual notion in the sense of Eckmann-Hilton. Given
$f:X\rightarrow Y$ and $g:X\rightarrow Z$ fibrewise maps we can
consider the quotient fibrewise space $C_{f,g}:=(Y\sqcup
I_B(X)\sqcup Z)/\hspace{-3pt}\sim $ where $\sim $ is the
equivalent relation generated by the elemental relations
$(x,0)\sim f(x)$ and $(x,1)\sim g(x),$ for all $x\in X.$ The
induced fibrewise map $w':C_{f,g}\rightarrow K$ is called the
\emph{fibrewise co-whisker map}. If $w'$ is a fibrewise homotopy
equivalence then the square of $(*)$ is called \emph{fibrewise
homotopy pushout}.

\begin{remark}
As in the classical case, fibrewise homotopy pullbacks and
fibrewise homotopy pushouts can be also characterized by the weak
universal property of fibrewise homotopy limits and colimits or
through factorization properties. The reader is referred to
\cite{M}, \cite{D} or \cite{B-K}.
\end{remark}

A combination of fibrewise homotopy pullbacks and fibrewise
homotopy pushouts is the \emph{fibrewise join} of two fibrewise
maps $f:X\rightarrow Z$ and $g:Y\rightarrow Z$. Namely, the
fibrewise join of $f$ and $g,$ $X*_Z Y,$ is the fibrewise homotopy
pushout of the fibrewise homotopy pullback of $f$ and $g,$
$$\xymatrix@C=0.7cm@R=0.7cm{ {\bullet } \ar[rr] \ar[dd] & & {Y} \ar[dl] \ar[dd]^g \\
 & {X*_Z Y} \ar@{.>}[dr] & \\ {X} \ar[ur] \ar[rr]_f & & {Z} }$$
being the dotted arrow the corresponding co-whisker map, induced
by the weak universal property of fibrewise homotopy pushouts.

The next result relates fibrewise cofibrations and fibrewise
fibrations and is one of the key tools in the proof of the Cube
Lemma.

\begin{lemma}\label{crucial}
Consider a fibrewise pullback of the following form, where $p$ is
a fibrewise fibration over $B$
$$\xymatrix{
{P} \ar[rr]^{j'} \ar[d]_{p'} & & {E} \ar[d]^p \\
{A} \ar[rr]_{j} & & {X} }$$ If $j:A\rightarrow X$ is a closed
fibrewise cofibration, then so is its base change $j':P\rightarrow
E.$
\end{lemma}

\begin{proof}
We can suppose that $(X,A)$ is a closed fibrewise cofibred pair,
and $j$ the natural inclusion, so that there exists $(\varphi ,H)$
a fibrewise Str{\o}m structure. Moreover, $P=A\times _X E=p^{-1}(A)$
with $j'$ the inclusion and $p'$ the corresponding restriction of
$p.$ Now take a lift in the diagram
$$\xymatrix{ {E} \ar[rr]^{id} \ar[d]_{i_0} & & {E} \ar[d]^p \\
{I_B(E)} \ar[rr]_{H(p\times id)} \ar@{.>}[urr]_{\overline{H}} & &
{X} }$$ Defining $\widetilde{H}(e,t)=\overline{H}(e,\min
\{t,\varphi p(e)\})$, then one can easily check that
$(\widetilde{H},\varphi p)$ is a fibrewise Str{\o}m structure for the
pair defined by the inclusion $j'.$ \end{proof}

\begin{theorem}[Cube Lemma]\label{cube} Given a fibrewise homotopy commutative cube
$$\xymatrix@!0{
  & {W } \ar[dl] \ar[rr] \ar'[d][dd]
      &  & {X } \ar[dd] \ar[dl]       \\
  {Y} \ar[rr]\ar[dd]
      &  & {Z} \ar[dd] \\
  & {A} \ar[dl] \ar'[r][rr]
      &  & {B}   \ar[dl]             \\
  {C} \ar[rr]
      &  & {D}         }$$
in which the bottom face is a fibrewise homotopy pushout and all
sides are fibrewise homotopy pullbacks, then the top face is also
a fibrewise homotopy pushout.
\end{theorem}

\begin{proof}
Following an analogous reasoning to the one given in \cite{M} we
can suppose without loss of generality that the cube is strictly
commutative in which:
\begin{itemize}
\item The arrows $A\rightarrow B$ and $A\rightarrow C$ are closed fibrewise
cofibrations, and the bottom face is a fibrewise pushout;
\item $Z\rightarrow D$ is a fibrewise fibration and all side squares are fibrewise pullbacks.
\end{itemize}
This way the top square is of the form
$$\xymatrix{
{A\times _DZ} \ar[rr] \ar[d] & & {B\times _DZ} \ar[d] \\
{C\times _DZ} \ar[rr] & & {Z} }$$ \noindent where, by Lemma
\ref{crucial} above, all the arrows are closed fibrewise
cofibrations. Now consider $P$ the fibrewise pushout of $C\times
_DZ{\longleftarrow }A\times _DZ{\longrightarrow }B\times _DZ$ and
$\theta :P\rightarrow Z$ the fibrewise map induced by the pushout
property. A simple inspection proves that $\theta $ is a fibrewise
isomorphism (compare with \cite[6.1]{D2} ) concluding that this
square is a fibrewise (homotopy) pushout. \end{proof}

\begin{remark}
Note that, even satisfying the Cube Lemma, $\mathbf{Top}_B$ is not
a $J$-category in the sense of Doeraene \cite{D} as it has no zero
object.
\end{remark}

\section{Fibrewise sectional category.}

And now we will establish the main notion of the paper in the
fibrewise context, the one of fibrewise sectional category. First
we give some background on fibrewise LS category and its unpointed
version.

Let $X$ be a fibrewise space. The \emph{fibrewise L.-S. category}
of $X$, $\mbox{cat}_B(X),$ is the minimal number $n\geq 0$ such
that there exists a cover $\{U_i\}_{i=0}^n$ of $X$ by $n+1$ open
subsets, each of them admitting a fibrewise homotopy commutative
diagram in $\mathbf{Top}_B$ of the form
$$\xymatrix{
{U_i} \ar@{^(->}[rr] \ar[dr]_{p_X|U_i} & & {X} \\
 & {B} \ar[ur]_{s_i} & }$$
If there is no such $n$, then we say $\mbox{cat}_B(X)=\infty .$
This notion was given by James-Morris in \cite{J-M} (see also
\cite{J2} and \cite{C-J}). Unfortunately, this invariant is not
very manageable from the axiomatic point of view. Instead we will
consider a certain variant of $\mbox{cat}_B(-),$ given by
Iwase-Sakai in \cite{I-S}. For this, we deal with fibrewise
pointed spaces. By a \emph{fibrewise pointed space over $B$} we
mean a fibrewise space $X$ together with a fibrewise map
$s_X:B\rightarrow X$ (i.e., $s_X:B\rightarrow X$ is a section of
$p_X$).

If $X$ is a fibrewise pointed space, then the \emph{fibrewise
unpointed L.-S. category} of $X$, $\mbox{cat}_B^*(X),$ is the
minimal number $n\geq 0$ such that there exists a cover
$\{U_i\}_{i=0}^n$ of $X$ by $n+1$ open subsets, each of them
\emph{fibrewise categorical}, that is, admitting a fibrewise
homotopy commutative diagram in $\mathbf{Top}_B$ of the form
$$\xymatrix{
{U_i} \ar@{^(->}[rr] \ar[dr]_{p_X|U_i} & & {X} \\
 & {B} \ar[ur]_{s_X} & }$$
If there is no such $n$, then we say $\mbox{cat}_B^*(X)=\infty .$

Obviously, $\mbox{cat}_B(X)\leq \mbox{cat}_B^*(X)$ and the
equality holds when $X$ is \emph{vertically connected}, that is,
all possible sections $s_i:B\rightarrow X$ are fibrewise homotopic
to $s_X$ (see \cite{J-M}). As James-Morris assert, when $B$ is a
CW-complex a fibre bundle over $B$ with fibre $F$ is vertically
connected if $\mbox{dim}(B)$ does not exceed the connectivity of
$F.$

\subsection{Open-like definition of fibrewise sectional category}

We want to give the natural generalization of fibrewise unpointed
LS category by considering the analogous notion of sectional
category in the fibrewise setting.

\begin{definition}{\rm
Let $f:E\rightarrow X$ be a fibrewise map over $B$ and consider an
open subset $U$ of $X.$ Then $U$ is said to be \emph{fibrewise
sectional} if there exists a morphism $s:U\rightarrow E$ in
$\mathbf{Top}_B$ such that the following triangle commutes up to
fibrewise homotopy
$$\xymatrix{ {U} \ar@{^(->}[rr]^{in} \ar[dr]_s & & {X} \\
 & {E} \ar[ur]_f & }$$ The \emph{fibrewise sectional category} of $f,$
denoted as $\mbox{secat}_B(f)$ is the minimal number $n$ such that
$X$ admits a cover $\{U_i\}_{i=0}^n$ constituted by fibrewise
sectional open subsets. If there is no such $n$, then
$\mbox{secat}_B(f)=\infty .$}
\end{definition}

When $f:E\rightarrow X$ is a fibrewise fibration, then we may
suppose in the definition of fibrewise sectional that the
triangles are strictly commutative. On the other hand, from the
definition of fibrewise sectional category it is clear the
identity
$$\mbox{cat}^*_B(X)=\mbox{secat}_B(s _X).$$ We can generalize this
fact. Observe that given a fibrewise homotopy commutative diagram
in $\mathbf{Top}_B$ of the form
$$\xymatrix{ {E} \ar[rr]^{\lambda } \ar[dr]_f & & {E'} \ar[dl]^{f'} \\
 & {X} & }$$ \noindent one has the inequality $\mbox{secat}_B(f')\leq \mbox{secat}_B(f).$
Indeed, if $U$ is an open fibrewise sectional categorical subset
of $X$ with local fibrewise homotopy section $s:U\rightarrow E$ of
$f,$ then $s'=\lambda s:U\rightarrow E'$ is a local fibrewise
homotopy section of $f'.$

A \emph{fibrewise contractible space} (or just a \emph{shrinkable}
space) is any fibrewise space having the fibrewise homotopy type
of $B.$ On the other hand if $X$ and $Y$ are fibrewise pointed
spaces over $B$, by a \emph{fibrewise pointed map} $f:X\rightarrow
Y$ we mean a fibrewise map such that $fs_X=s_Y.$ Applying the
above comments to the commutative triangle $fs_E=s_X$ we obtain

\begin{proposition}
Let $f:E\rightarrow X$ be any fibrewise pointed map. Then
$$\mbox{secat}_B(f)\leq \mbox{cat}_B^*(X)$$
If $E$ is fibrewise contractible, then
$\mbox{secat}_B(f)=\mbox{cat}_B^*(X).$
\end{proposition}

One can also check that the fibrewise unpointed LS category
$\mbox{cat}^*_B(-)$ is invariant by fibrewise pointed maps which
are fibrewise homotopy equivalences.

\begin{proposition}\label{invariance}
If $f:X\rightarrow Y$ is a fibrewise pointed map such that it is a
fibrewise homotopy equivalence, then
$$\mbox{cat}^*_B(X)=\mbox{cat}^*_B(Y).$$
In particular, if $X$ is a fibrewise pointed space over $B,$ then
$\mbox{cat}^*_B(X)=0$ if and only if $X$ is fibrewise
contractible.
\end{proposition}

\subsection{The axiomatic approach of fibrewise sectional category}

Now we study the fibrewise sectional category $\mbox{secat}_B(-)$
from a Whitehead-Ganea approach. Let $f:E\rightarrow X$ be a
fibrewise map. For each $n\geq 0$ we consider the \emph{fibrewise
sectional n-fat wedge} as the fibrewise map
$j_n:T_B^n(f)\rightarrow \prod _B^{n+1}X$ inductively defined as
follows: set $T^0_B(f)=E$, $j_0=f:E\rightarrow X$, and define
$j_{n}:T^{n}_B(f)\rightarrow \prod _B^{n+1}X$ as the join in
$\mathbf{Top}_B$
$$\xymatrix@C=0.5cm@R=0.6cm{ {\bullet } \ar[rr] \ar[dd] & & {X\times _BT^{n-1}_B(f)}
\ar[dl] \ar[dd]^{id_X \times _Bj_{n-1}} \\
 & {T^n_B(f)} \ar@{.>}[dr]^{j_n} & \\ {B\times _B\prod _B^{n}X} \ar[ur] \ar[rr]_{f \times _Bid} & &
{\prod _B^{n+1}X} }$$

On the other hand, the $n$-th \emph{fibrewise Ganea map} of $X,$
$G^n_B(X)\stackrel{p_n}{\longrightarrow }X,$ is inductively
defined as follows:

Set $p_0:=f:E\rightarrow X$ (so $G^0_B(f)=E$).  If $p_{n-1}$ is
already constructed, then $G^{n}_B(f)$ is the fibrewise join of
$G^{n-1}_B(f)\stackrel{p_{n-1}}{\longrightarrow }X\stackrel{f
}{\longleftarrow }E$, and $p_n$ is the induced whisker map:
$$\xymatrix@C=0.5cm@R=0.6cm{ {\bullet } \ar[rr] \ar[dd] & & {E}
\ar[dl] \ar[dd]^{f} \\
 & {G^{n}_B(f)} \ar@{.>}[dr]^{p_n} & \\ {G^{n-1}_B(f)} \ar[ur] \ar[rr]_{p_{n-1}} & & {X} }$$

The following result is a direct consequence of Theorem
\ref{cube}, the Cube Lemma in the fibrewise context. Its proof is
similar to the classical case (see, for instance, \cite{C-L-O-T})
and therefore is omitted and left to the reader.

\begin{lemma}\label{WG}
Let $f:E\rightarrow X$ be a fibrewise map. Then, for each $n\geq
0,$ there is a fibrewise homotopy pullback
$$\xymatrix{
  {G^n_B(f)} \ar[d]_{p_n} \ar[r]
                & {T^n_B(f)} \ar[d]^{j_n}  \\
  {X}  \ar[r]_{\Delta _{n+1}}
                & {\prod _B^{n+1}X}             }$$
\noindent where $\Delta _{n+1}$ denotes the $n+1$-st diagonal map.
\end{lemma}

\begin{lemma}\label{fat-cof}
Let $f:E\rightarrow X$ be a closed fibrewise cofibration. Then,
the fibrewise sectional n-fat wedge is, up to fibrewise homotopy,
\begin{center}$T^n_B(f)=\{(x_0,x_1,...,x_n)\in \prod
_B^{n+1}X\hspace{3pt};\hspace{3pt}x_i\in E \hspace{4pt}\mbox{for
some}\hspace{4pt}i\in \{0,1,...,n\}\}$\end{center} \noindent being
$j_n:T^n_B(f)\hookrightarrow \prod _B^{n+1}X$ the canonical
inclusion.
\end{lemma}

\begin{proof} It is evident for $n=0.$ Next, we  check that the following square is a fibrewise homotopy pullback:
$$\xymatrix{
  {E\times _BT^{n-1}_B(f)} \ar@{_{(}->}[d]_{id\times _Bj_{n-1}}
   \ar[rr]^{f \times _Bid}
              &  & {X\times _BT^{n-1}_B(f)}
                \ar@{_{(}->}[d]_{id\times _Bj_{n-1}}  \\
  {E\times _B\prod _B^{n}X} \ar[rr]_{f \times _B id}
              &  & {\prod _B^{n+1}X}             }$$
The standard fibrewise homotopy pullback of $f \times _Bid$ and
$id_X\times _Bj_{n-1}$, call it $L$, is given by the elements
\begin{center}$(e,\bar{x},(b,\gamma ),x,\bar{y})\in (E\times _B
\prod _B^{n}X)\times P_B(\prod _B^{n+1}X)\times (X\times _B
T^{n-1}_B(f))$\end{center} \noindent for which $\gamma
(0)=(f(e),\bar{x})$ and $\gamma (1)=(x,\bar{y})$. Define $\omega
:E\times _BT^{n-1}_B(f)\rightarrow L$ by
$\omega(e,\bar{y})=(e,\bar{y},(b,C_{(f(e),\bar{y})}),f(e),\bar{y})$
(where $C_{(f(e),\bar{y})}$ denotes the constant path in
$(f(e),\bar{y})$), and $\omega ':L\rightarrow E\times _B
T^{n-1}_B(f)$ by $\omega '(e,\bar{x},(b,\gamma
),x,\bar{y})=(e,\bar{y}).$ Then, $\omega '\omega=id$ and $\omega
\omega '\simeq _Bid$ through the fibrewise homotopy
\begin{center}$H(e ,\bar{x},(b,\gamma ),x,\bar{y};t)=(e,\gamma
_2(t),(b,\delta (t)), \gamma _1(1-t),\bar{y})$,\end{center} being
$\gamma =(\gamma _1,\gamma _2)$ and $\delta (t)(s)=(\gamma
_1(s(1-t)),\gamma _2((1-s)t+s)).$

Finally, taking into account that $f\times _B id:E\times _B
T^{n-1}_B(f)\rightarrow X\times _BT^{n-1}_B(X)$ is a closed
fibrewise cofibration, the fibrewise homotopy pushout of $f\times
_Bid$ and the inclusion $id\times _Bj_{n-1}$ is just its honest
fibrewise pushout. The result follows by induction.

\end{proof}

\begin{remark}
If $f$ is any fibrewise map, then we can factor it through a
closed fibrewise cofibration followed by a fibrewise homotopy
equivalence
$$\xymatrix{
{E} \ar[rr]^{f} \ar[dr]_{f'} & & {X} \\
 & {X'} \ar[ur]^{\simeq } &
}$$ Therefore any fibrewise map can be considered, up to fibrewise
homotopy, as a closed fibrewise cofibration. Moreover, as this
factorization can be taken through the fibrewise mapping cylinder,
if $E$ and $X$ are normal, then $X'$ is also normal.
\end{remark}

\begin{theorem}\label{three-notions}
Let $f:E\rightarrow X$ be any fibrewise map between normal spaces,
or a closed fibrewise cofibration with $X$ normal. Then the
following statements are equivalent:
\begin{enumerate}
\item[(i)] $\mbox{secat}_B(f)\leq n$

\item[(ii)] The diagonal map $\Delta _{n+1}:X\rightarrow \prod
_B^{n+1}X$ factors, up to fibrewise homotopy, through the
fibrewise sectional n-fat wedge
$$\xymatrix{ {X} \ar[r] \ar[dr]_{\Delta _{n+1}} & {T^n_B(f)} \ar[d]^{j_n} \\ & {\prod _B^{n+1}X} }$$

\item[(iii)] The $n$-th fibrewise Ganea map
$p_n:G^n_B(f)\rightarrow X$ admits a fibrewise homotopy section.
\end{enumerate}
\end{theorem}

\begin{proof}
Statements (ii) and (iii) are equivalent as a consequence of Lemma
\ref{WG} and the weak universal property for fibrewise homotopy
pullbacks. Now we check that statements (i) and (ii) are
equivalent. Taking into account the above remark and the fact that
statements (i)-(iii) are invariant by fibrewise homotopy
equivalences, we can suppose without loss of generality that $f$
is a fibrewise closed cofibration with $X$ normal.

Assume that $\mbox{secat}_B(f)\leq n$ and consider
$\{U_i\}_{i=0}^n$ an open cover of $X$ and
$H_i:I_B(U_i)\rightarrow X$ a fibrewise homotopy satisfying
$H_i(x,0)=x$ and $H_i(x,1)=fs_i(x),$ for all $x\in U_i,$ where
$s_i:U_i\rightarrow E$ is a fibrewise map. As $X$ is a normal
space there exist, for each $i$, closed subsets $A_i,B_i$ and an
open subset $\Theta _i$ such that $A_i\subseteq \Theta _i\subseteq
B_i\subseteq U_i$ and $\{A_i\}_{i=0}^n$ covers $X.$ Now, by
Urysohn characterization of normality, take $h_i:X\rightarrow I$ a
map such that $h_i(A)=\{1\}$ and $h_i(X\setminus \Theta
_i)=\{0\}.$ Then we obtain a fibrewise homotopy
$L=(L_0,...,L_n):I_B(X)\rightarrow \prod ^{n+1}_BX$, where
$L_i:I_B(X)\rightarrow X$ is the fibrewise map defined as follows:
$$L_i(x,t)=\left\{\begin{array}{ll} x  & ,x\in X\setminus B_i \\
{H_i(x,th_i(x))} & ,x\in U_i \end{array}\right.$$ Taking into
account Lemma \ref{fat-cof} and that $\{A_i\}_{i=0}^n$ covers $X$
it is straightforward to check that $H:\Delta _{n+1}\simeq _B
j_n\varphi,$ where $\varphi:X \rightarrow T^n_B(f)$ is defined as
$\varphi (x):=L(x,1).$

Conversely, suppose a fibrewise map $\varphi:X\rightarrow
T^n_B(f)$ and a fibrewise homotopy $L:\Delta _{n+1}\simeq _B
j_n\varphi.$ Then, $L=(L_0,...,L_n)$ and $j_n\varphi=(\varphi
_0,...,\varphi _n)$ where $L_i:I_B(X)\rightarrow X$ and
$\varphi_i:X\rightarrow X$ for each $i.$ Since $f:E\hookrightarrow
X$ is fibrewise closed cofibration, by Remark \ref{rem2} we have
that $E$ is a fibrewise strong deformation retract of an open
neighborhood $U$ in $X.$ Take a fibrewise retraction
$r:U\rightarrow E$ and a fibrewise homotopy $H:I_B(U)\rightarrow
X$ such that $H(x,0)=x,$ $H(x,1)=fr(x)$ for all $x\in U,$ and
$H(e,t)=e=f(e),$ for all $e\in E$ and $t\in I.$ Defining
$U_i=\varphi _i^{-1}(U)$ we obtain $\{U_i\}_{i=0}^n$ an open cover
of $X$ and fibrewise homotopies $G_i:I_B(U_i)\rightarrow X$
$$G_i(x,t)=\left\{\begin{array}{ll} {L_i(x,2t)} &, 0\leq t\leq \frac{1}{2}
\\[.2pc]
{H(\varphi_i(x),2t-1)} &, \frac{1}{2}\leq t\leq 1 \\[.2pc]
\end{array}\right.$$ \noindent If $s_i:U_i\rightarrow E$ denotes the composite
$U_i\stackrel{\varphi _i}{\longrightarrow
}{U}\stackrel{r}{\longrightarrow }E$, then $G_i(x,0)=x$ and
$G_i(x,1)=fs_i(x),$ for all $x\in U_i.$
\end{proof}

\begin{remark}
Observe that taking $f=s_X$ in the above theorem we obtain the
corresponding Whitehead-Ganea characterization of
$\mbox{cat}^*_B(-).$ Compare with \cite{I-S}.
\end{remark}

\subsection{Fibrewise sectional category and ordinary sectional category.}

Recall that if $f:X\rightarrow Y$ is \emph{any} map, then the
\emph{sectional category} $\mbox{secat}(f)$ can be defined as the
least non-negative integer $n$ such that $Y$ admits an open cover
$\{U_i\}_{i=0}^n$ where each $U_i$ has a local homotopy section
$s_i:U_i\rightarrow X$ of $f.$ When $f$ is a Hurewicz fibration,
then in this definition we can take local strict sections,
recovering the usual notion of sectional category (or Schwarz
genus \cite{Sch}) for fibrations.

On the other hand, one of the main motivations of the invariant
$\mbox{cat}_B^*(-)$ is the fact that it retrieves the notion of
\emph{topological complexity} of a space in the sense of Farber
\cite{F,F2}. The topological complexity of a space $X$, denoted
$\mbox{TC}(X)$, is defined as the sectional category of $\pi
:X^I\rightarrow X\times X,$ $\alpha \mapsto (\alpha (0),\alpha
(1)).$ Of course, since the fibration $\pi :X^I\rightarrow X\times
X$ is homotopy equivalent to the diagonal map $\Delta
_X:X\rightarrow X\times X,$ then we can redefine $\mbox{TC}(X)$ as
$\mbox{secat}(\Delta _X).$ The product space $X\times X$ can be
seen as a fibrewise pointed space over $X$ with $\Delta _X$ the
diagonal map as the section and $pr_2:X\times X\rightarrow X$ the
projection. Denoting by $d(X)$ such a fibrewise pointed space we
have the following known equality (see \cite{I-S}):
$$\mbox{TC}(X)=\mbox{cat}_X^*(d(X))$$
This fact can be generalized for the unpointed LS category. Under
a not very restrictive condition on the fibrewise pointed space
$X$ we will be able to prove that
$\mbox{cat}_B^*(X)=\mbox{secat}(s_X).$

By a \emph{fibrant space} over $B$ we mean a fibrewise space over
$B$, $X$, such that the projection $p_X:X\rightarrow B$ is a
Hurewicz fibration. A \emph{pointed fibrant space} over $B$ is a
fibrant space over $B$ which is also pointed over $B.$

\begin{theorem}\label{secat}
Let $f:E\rightarrow X$ be a fibrewise map between fibrant spaces
over $B.$ Then
$$\mbox{secat}_B(f)=\mbox{secat}(f)$$
\end{theorem}

\begin{proof}
The inequality $\mbox{secat}(f)\leq \mbox{secat}_B(f)$ is obvious
so it only remains to check $\mbox{secat}_B(f)\leq
\mbox{secat}(f).$ First consider a factorization of $f$ in
$\mathbf{Top}$
$$\xymatrix{
{E} \ar[rr]^{f} \ar[dr]_{\lambda } & & {X} \\
 & {\widehat{E}} \ar[ur]_{p} & }$$ \noindent by a
homotopy equivalence $\lambda $ followed by a Hurewicz fibration
$p.$ We observe that this factorization also has sense in
$\mathbf{Top}_B$ considering in $\widehat{E}$ the composite
$\widehat{E}\stackrel{p}{\longrightarrow
}X\stackrel{p_X}{\longrightarrow }B$ as the natural projection,
which is a Hurewicz fibration. Being $E$ and $\widehat{E}$ fibrant
spaces over $B$ we conclude that actually $\lambda :E\rightarrow
\widehat{E}$ is a fibrewise homotopy equivalence. Furthermore,
$\mbox{secat}(f)=\mbox{secat}(p)$ and
$\mbox{secat}_B(f)=\mbox{secat}_B(p).$

If $U$ is any open subset of $X$ with a strict local section
$s:U\rightarrow \widehat{E}$ of $p,$ then $s$ is necessarily a
fibrewise map over $B$ and therefore $U$ is also an open fibrewise
sectional subset of $X.$ The result follows taking open coverings.
\end{proof}

\begin{corollary}\label{secat2}
Let $X$ be a pointed fibrant space over $B.$  Then
$$\mbox{cat}_B^*(X)=\mbox{secat}(s_X).$$
\end{corollary}

As an interesting and surprising consequence of Theorem
\ref{secat} and its corollary we have that, in order to deal with
$\mbox{secat}_B(f)$ or $\mbox{cat}^*_B(X)$ we may forget the
projections and work in the non-fibrewise world, at least for weak
fibrant spaces.

\begin{remark}Observe that, in general, the above equality is not true. For
example, consider $X=([0,1]\times \{0,1\})\cup (\{0\}\times
[0,1])$ and $B=[0,1]$ with $s_X(t)=(t,0)$ and $p_X(t,s)=t.$ If $X$
were fibrewise contractible, then the fibre $p_X^{-1}(\{t\})$
would have the same ordinary homotopy type of $\{t\},$ for all
$t\in [0,1]$; but by connectivity reasons this is not the case.
Therefore, by Corollary \ref{invariance} we have that
$\mbox{cat}^*_B(X)\geq 1.$ But $\mbox{secat}(s_X)=0,$ since $X$ is
contractible in the ordinary sense.\end{remark}

The hypothesis of Theorem \ref{secat} can be relaxed. Recall that
a \emph{Dold fibration} (or \emph{weak fibration}) is a map
$p:X\rightarrow B$ which satisfies the weak covering homotopy
property (see \cite{Dld2} for details). A Dold fibration is
completely characterized by the fact that it is fibrewise homotopy
equivalent to a Hurewicz fibration. Then we can consider
\emph{weak fibrant spaces} over $B,$ i.e., fibrewise spaces over
$B$ in which the projection is a Dold fibration. It is immediate
to check the equality $\mbox{secat}_B(f)=\mbox{secat}(f)$ when $E$
and $X$ are just weak fibrant spaces over $B,$ and the equality
$\mbox{cat}_B^*(X)=\mbox{secat}(s_X)$ when $X$ is a pointed weak
fibrant space over $B.$

\section{The fibrewise pointed case.}

The category of fibrewise pointed spaces and fibrewise pointed
maps will be denoted by $\mathbf{Top}(B).$ Observe that the space
$B$ together with the identity is the zero object so that
$\mathbf{Top}(B)$ is a pointed category. Any subspace $A\subseteq
X$ \emph{containing the section} (i.e., $s_X(B)\subseteq A$) is a
fibrewise pointed space. This way, the inclusion $A\hookrightarrow
X$ is a fibrewise pointed map. These class of subspaces will be
called \emph{fibrewise pointed subsets} of $X.$

For any fibrewise pointed space $X$ we can consider its
\emph{fibrewise pointed cylinder} as the following pushout
$$\xymatrix{
{B\times I} \ar[r]^{pr} \ar[d]_{s_X\times id} & {B} \ar[d] \\
{X\times I} \ar[r]  & {I_B^B(X)} }$$ \noindent with the obvious
projection induced by the pushout property. This pointed cylinder
functor gives the notion of \emph{fibrewise pointed homotopy}
between fibrewise pointed maps, that will be denoted by $\simeq
_B^B.$ We point out that giving a fibrewise pointed homotopy
$F:I_B^B(X)\rightarrow Y$ is the same as giving a fibrewise
homotopy $F':I_B(X)\rightarrow Y$ such that $F'(s_X(b),t)=s_X(b),$
for all $b\in B$ and $t\in I.$ The notion of \emph{fibrewise
pointed homotopy equivalence} comes naturally. On the other hand
we can consider $$P_B^B(X)=B\times _{B^I}X^I=\{(b,\alpha )\in
B\times X^I:c_b=p_X\alpha \}$$ \noindent which is simply the
fibrewise space $P_B(X)$ together with the section
$(id_B,cs_X):B\rightarrow P_B^B(X)$ induced by the pullback
property. There are functors
$$I_B^B,P_B^B:\mathbf{Top}(B)\rightarrow \mathbf{Top}(B)$$
\noindent as well as induced natural transformations
$i_0,i_1:X\rightarrow I^B_B(X)$, $\rho :I^B_B(X)\rightarrow X$ and
$d_0,d_1:P_B^B(X)\rightarrow X,$ $c:X\rightarrow P_B^B.$ Moreover,
$(I_B^B,P_B^B)$ is an adjoint pair in the sense of Baues.

Associated to these functors, there are defined the notions of
\emph{(closed) fibrewise pointed cofibration} and of
\emph{fibrewise pointed fibration}, which are characterized by the
natural Homotopy Extension Property and the Homotopy Lifting
Property in $\mathbf{Top}(B),$ respectively. Moreover,
$\mathbf{Top}(B)$ has an $I$-category and a $P$-category structure
in the sense of Baues (\cite[p.31]{B}).

Any fibrewise pointed map which is a fibrewise cofibration is a
fibrewise pointed cofibration. Also, any fibrewise fibration is a
fibrewise pointed fibration. Unfortunately, as P. May and
Sigurdsson assert in \cite[p.82]{M-S}, on each case the converse
is not true even for the simplest case, in which $B$ is a point.
Despite this fact, in order to deal with LS-type invariants in the
homotopy category of $\mathbf{Top}(B),$ we still can achieve a
reasonable structure on the category of fibrewise pointed spaces
based one the fibrewise structure.

A fibrewise \emph{well-pointed} space is a fibrewise pointed space
$X$ in which the section $s_X:B\rightarrow X$ is a closed
fibrewise cofibration. Let $\mathbf{Top}_w(B)$ denote the full
subcategory of $\mathbf{Top}(B)$ consisting of fibrewise
well-pointed spaces. One cannot expect $\mathbf{Top}_w(B)$ to be a
model category as it is not closed under finite limits and
colimits. However the following proposition is enough for our
purposes. First we need a technical lemma, whose proof is
analogous to the ordinary case and therefore is omitted (see
\cite{S}).

\begin{lemma}\label{cof-comp}
Let $j:A\rightarrow X$ and $i:X\rightarrow Y$ be fibrewise maps
such that $i$ and $ij$ are closed fibrewise cofibrations. Then $j$
is also a closed fibrewise cofibration.
\end{lemma}

\begin{proposition}\label{importante2}
$\mathbf{Top}_w(B)$ is closed under the pullbacks of fibrewise
pointed maps which are fibrewise fibrations. Similarly,
$\mathbf{Top}_w(B)$ is closed under the pushouts of fibrewise
pointed maps which are closed fibrewise cofibrations.
\end{proposition}

\begin{proof}
As the cobase change of a fibrewise cofibration is a fibrewise
cofibration, the second statement of the lemma is trivially true.
Now suppose $p:E\rightarrow X$ and $f:X'\rightarrow X$ fibrewise
pointed maps between fibrewise well-pointed spaces where $p$ is a
fibrewise fibration. Consider the following diagram of pullbacks
$$\xymatrix{
{F} \ar[d] \ar[r]^i & {E'} \ar[d]_{p'} \ar[r]^{f'} & {E} \ar[d]^p \\
{B} \ar[r]_{s_{X'}} & {X'} \ar[r]_{f} & {X} }$$ As $X'$ and $X$
are well-pointed and $p$ is a fibrewise fibration, by Lemma
\ref{crucial} $i$ and $f'i$ are closed fibrewise cofibrations.
Considering Proposition \ref{cof-comp} and the fact that
$(f'i)s_{F}=s_E$ is a closed fibrewise cofibration we have that
$F$ is well-pointed and therefore $s_{E'}=is_F$ is a closed
fibrewise cofibration.
\end{proof}

A certain version of the following result also appears in
\cite[Prop 5.2.2, Prop 5.2.3]{M-S}.

\begin{proposition}\label{importante1}
Let $f:X\rightarrow Y$ be a fibrewise pointed map between
fibrewise well-pointed spaces over $B.$ Then,

\begin{enumerate}
\item[(i)] $f$ is a fibrewise pointed fibration if and only if $f$ is a fibrewise
fibration;

\item[(ii)] If $f$ is a closed map, then $f$ is a fibrewise pointed cofibration if and only if $f$ is a fibrewise
cofibration;

\item[(iii)] $f$ is a fibrewise pointed homotopy equivalence if and only if $f$ is a fibrewise homotopy
equivalence.
\end{enumerate}
\end{proposition}

\begin{proof}
Part (i) is given in \cite[Prop 16.3]{C-J}, while part (iii) is
consequence of the abstract Dold's theorem in the cofibration
category $\mathbf{Top}_B$ (\cite[p.96]{B}). A simple inspection
reveals that the reasonings given in \cite{S} can be
straightforwardly extrapolated to the fibrewise setting for the
proof of part (ii).
\end{proof}

These facts lead to a great simplification in this setting and
allow us to prove directly the following result.

\begin{theorem}\label{J-category}
The category $\mathbf{Top}_w(B)$ with the structure inherited by
$\mathbf{Top}_B$ is a $J$-category in the
sense of Doeraene \cite{D}. %in which all objects are cofibrant
\end{theorem}

\begin{proof}
Considering the above two propositions combined with Proposition
\ref{prev} we have that $\mathbf{Top}_B$ induces in
$\mathbf{Top}_w(B)$ a cofibration and fibration category
structures in the sense of Baues. Observe that the factorizations
of a fibrewise pointed map are taken through the fibrewise mapping
cylinder or the fibrewise mapping track. Again, propositions
\ref{importante1} and \ref{importante2} together with Theorem
\ref{cube} show that $\mathbf{Top}_w(B)$ satisfy the Cube Lemma.
\end{proof}

\begin{definition}{\rm
Let $f:E\rightarrow X$ be a fibrewise map over $B$ and consider an
open subset $U$ of $X$ containing $s_X(B)$. Then $U$ is said to be
\emph{fibrewise pointed sectional} if there exists a morphism
$s:U\rightarrow E$ in $\mathbf{Top}(B)$ such that the following
triangle commutes up to fibrewise pointed homotopy
$$\xymatrix{ {U} \ar@{^(->}[rr]^{in} \ar[dr]_s & & {X} \\
 & {E} \ar[ur]_f & }$$ The \emph{fibrewise pointed sectional category} of $f,$
denoted as $\mbox{secat}^B_B(f)$ is the minimal number $n$ such
that $X$ admits an open cover $\{U_i\}_{i=0}^n$ constituted by
fibrewise pointed sectional subsets. When such $n$ does not exist
then $\mbox{secat}^B_B(f)=\infty .$ }
\end{definition}

When $f$ is a fibrewise pointed fibration, having a local strict
section of $f$ is equivalent to having a local fibrewise pointed
homotopy section of $f.$ On the other hand it is immediate to
check the identity $$\mbox{secat}^B_B(s _X)=\mbox{cat}^B_B(X),$$
\noindent that is, the pointed fibrewise sectional category of
$s_X$ is exactly the fibrewise pointed LS category in the sense of
James-Morris \cite{J-M}.

A \emph{fibrewise pointed contractible space} is any fibrewise
pointed space having the same fibrewise pointed homotopy type of
$B.$ The proof of the following result is analogous to the
fibrewise unpointed case.

\begin{proposition}
Let $f:E\rightarrow X$ be any fibrewise pointed map. Then
$$\mbox{secat}^B_B(f)\leq \mbox{cat}_B^B(X)$$
If $E$ is fibrewise pointed contractible, then
$\mbox{secat}^B_B(f)=\mbox{cat}_B^B(X).$
\end{proposition}

As a consequence of Theorem \ref{J-category}, by \cite{D-GC-R-M-R}
or \cite{K} we can define a manageable axiomatic notion of
fibrewise pointed sectional category from two equivalent
approaches: Whitehead's and Ganea's. Moreover, the \emph{fibrewise
pointed n-fat wedge} and the \emph{n-th fibrewise pointed Ganea
fibration} can be chosen as $j_n:T^n_B(f)\rightarrow \prod
^{n+1}_BX$ and $p_n:G^n_B(f)\rightarrow X,$ the ones given in the
unpointed case. The proof of the following theorem is completely
analogous to the one given in Theorem \ref{three-notions} and
therefore is omitted.
%For its statement one has just to replace the words
%\emph{fibrewise map} by \emph{fibrewise pointed map}, and
%\emph{fibrewise homotopy} by \emph{fibrewise pointed homotopy}.
For the particular case $f=s_X$ compare the equivalence of (i) and
(ii) in our theorem with \cite[Prop 6.1 and Prop 6.2]{J-M}

\begin{theorem}
Let $f:E\rightarrow X$ be a fibrewise pointed map in
$\mathbf{Top}_w(B)$ between normal spaces, or a closed fibrewise
pointed cofibration with $X$ normal. Then the following statements
are equivalent:
\begin{enumerate}
\item[(i)] $\mbox{secat}_B^B(f)\leq n$

\item[(ii)] The diagonal map $\Delta _{n+1}:X\rightarrow \prod
_B^{n+1}X$ factors, up to fibrewise pointed homotopy, through the
fibrewise n-fat wedge
$$\xymatrix{ {X} \ar[r] \ar[dr]_{\Delta _n} & {T^n_B(f)} \ar[d]^{j_n} \\ & {\prod _B^{n+1}X} }$$

\item[(iii)] The $n$-th fibrewise Ganea map
$p_n:G^n_B(f)\rightarrow X$ admits a fibrewise pointed homotopy
section.
\end{enumerate}
\end{theorem}

\begin{remark}
As in the unpointed case, observe that taking $f=s_X$ in the above
theorem we obtain a Whitehead-Ganea characterization of
$\mbox{cat}^B_B(-).$ Compare with \cite{I-S}.
\end{remark}

We conclude this section with some comments and remarks. Recall
from \cite{I-S} that the monoidal topological complexity of a
space $X$, denoted $\mbox{TC}^M(X)$, is the least integer $n$ such
that $X\times X$ can be covered by $n+1$ open sets
$\{U_i\}_{i=0}^n$ on each of which there is a local section $s_i$
of the free path fibration
$$\pi _X:X^I\rightarrow X\times X,\hspace{15pt}\pi _X(\gamma )=(\gamma (0),\gamma
(1))$$ satisfying that $\Delta _X(X)\subset U_i$ and
$s_i(x,x)=c_x,$ for all $x\in X.$

As we have previously commented the product space $X\times X$ can
be seen as a fibrewise pointed space over $X,$ $d(X),$ with
$\Delta _X$ the diagonal map as the section and $pr_2:X\times
X\rightarrow X$ the projection. We have the equality
$$\mbox{TC}^M(X)=\mbox{cat}_X^X(d(X))$$
\noindent for any space $X$ (see \cite{I-S} for details). In
particular $\mbox{TC}^M(X)=\mbox{secat}_X^X(\Delta _X).$ We assert
that the monoidal topological complexity can also be seen as
$$\mbox{TC}^M(X)=\mbox{secat}^X_X(\pi _X).$$
Indeed, considering $X^I$ as a fibrewise pointed space over $X$
with section $c:X\rightarrow X^I$ sending $x\mapsto c_x$ ($c_x$ is
the constant path at $x$) and projection
$\mbox{ev}_1:X^I\rightarrow X$ the evaluation map $\alpha\mapsto
\alpha (1),$ we have that in the following commutative diagram in
$\mathbf{Top}(X)$
$$\xymatrix{ {X} \ar[rr]^{c} \ar[dr]_{\Delta _X} & & {X^I} \ar[dl]^{\pi _X} \\
 & {X\times X} & }$$
\noindent the homotopy equivalence $c:X\rightarrow X^I$ is
actually a fibrewise pointed homotopy equivalence over $X.$
Trivially $ev_1\hspace{2pt}c=id_X,$ and the formula $H(\alpha
,t)(s)=\alpha ((1-t)s+t)$ defines a fibrewise pointed homotopy
over $X$ between $id_{X^I}$ and the composite $c\hspace{2pt}ev_1.$

\section{Unpointed versus pointed fibrewise sectional category.}

In this last section we will check how close $\mbox{secat}_B(-)$
and $\mbox{secat}_B^B(-)$ are. Obviously, the inequality
$\mbox{secat}_B(-)\leq \mbox{secat}_B^B(-)$ always holds.

\begin{theorem}\label{first-condition}
Let $f:E\rightarrow X$ be a fibrewise pointed map in
$\mathbf{Top}_w(B)$ between normal spaces, or a closed fibrewise
cofibration with $X$ normal. Then
$$\mbox{secat}_B(f)\leq \mbox{secat}_B^B(f)\leq \mbox{secat}_B(f)+1$$
\end{theorem}

\begin{proof}
We can suppose without loss of generality that $f:E\hookrightarrow
X$ is a closed fibrewise cofibration with $X$ normal. Therefore,
by Corollary \ref{rem2}, $E$ is a fibrewise strong deformation
retract of an open neighborhood $N$ in $X.$ Take $r:N\rightarrow
E$ a fibrewise retraction and a fibrewise homotopy
$G:I_BN\rightarrow X$ with $G(x,0)=x,$ $G(x,1)=fr(x),$ for all
$x\in N$ and $G(e,t)=e,$ for all $e\in E$ and $t\in I.$

Suppose that $\mbox{secat}_B(f)=n$ and take $\{U_i\}_{i=0}^n$ an
open cover of $X,$ where each $U_i$ is fibrewise sectional with
$s_i:U_i\rightarrow E$ the local fibrewise homotopy section of
$f.$ For each $i\in \{0,...,n\}$ we choose $H_i:I_BU_i\rightarrow
X$ a fibrewise homotopy satisfying $H_i(x,0)=x$ and
$H_i(x,1)=fs_i(x),$ for all $x\in U_i.$

If we consider $V_i=U_i\setminus E$ for $i\in\{0,...,n\}$ and
$V_{n+1}=N,$ then we have that $\{V_i\}_{i=0}^{n+1}$ is an open
cover of $X$, and by normality one can find $\{W_i\}_{i=0}^{n+1}$
a refined open cover of $X$ such that $W_i\subseteq
\overline{W_i}\subseteq V_i,$ for all $i\in \{0,...,n+1\}.$

Now we define the open subset $$\mathcal{N}=N\cap (X\setminus
\overline{W_0})\cap ...\cap (X\setminus \overline{W_n})$$
Obviously, $\mathcal{N}\cap W_i=\emptyset $ and $E\subseteq
\mathcal{N},$ for all $i\in \{0,...,n\}.$

If $O_i=W_i\cup \mathcal{N},$ for $i\in \{0,...,n\}$ and
$O_{n+1}=N$, then $\{O_i\}_{i=0}^{n+1}$ is a cover of $X$
constituted by $n+2$ open fibrewise pointed categorical subsets
(observe that $W_{n+1}\subseteq V_{n+1}=N$). Indeed, it only
remains to check that $O_i$ is fibrewise pointed sectional, for
$i\neq n+1.$ In this case the following fibrewise pointed homotopy
$L_i:I_B(O_i)\rightarrow X$ proves this fact
$$L_i(x,t)=\left\{\begin{array}{ll}
{H_i(x,t),} & {x\in W_i} \\
{G(x,t),} & {x\in \mathcal{N}}
\end{array}\right.$$
\end{proof}

\begin{remark}
Observe that Theorem \ref{first-condition} is also true when we
consider fibrewise pointed embeddings $f:E\hookrightarrow X$ with
$X$ normal and such that $f$ admits an open fibrewise pointed
sectional subset $U\subseteq X.$ For instance, when $X$ is an ENR
(Euclidean Neighborhood Retract) then $\Delta _X:X\rightarrow
X\times X$ is known to satisfy this condition.
\end{remark}

As a corollary of the above result we have that if $X$ is any
normal fibrewise well-pointed space over $B,$ then
$$\mbox{cat}_B^*(X)\leq \mbox{cat}_B^B(X)\leq \mbox{cat}_B^*(X)+1$$
In particular, if $X$ is a locally finite simplicial complex (or
more generally, an ENR), then
$$\mbox{TC}(X)\leq \mbox{TC}^M(X)\leq \mbox{TC}(X)+1.$$
These two latter results have been already proved in \cite{I-S-2}
and \cite{Dr}.

\bigskip

In order to establish the statement of our second theorem we need
the following lemma whose proof can be found in \cite[Th
24.1.2]{M-S}. Here the bracket $[Z,K]_B$ denotes the set of
fibrewise homotopy classes of fibrewise maps $Z\rightarrow K$ over
$B.$

\begin{lemma}\label{fibreWhite}
Let $X$ be a CW-complex over $B$ and $e:Y\rightarrow Z$ an
n-equivalence over $B$ between fibrant spaces. Then
$$e_*:[X,Y]_B\rightarrow [X,Z]_B;\hspace{8pt}[\alpha ]\mapsto [e\alpha ]$$
\noindent is a bijection if $\mbox{dim}(X)<n$ and a surjection if
$\mbox{dim}(X)=n.$
\end{lemma}

If $X$ is a fibrewise pointed space over $B,$ then we say that $X$
is \emph{cofibrant} when the section $s_X$ is a closed cofibration
in $\mathbf{Top}.$ Observe the difference with the notion of being
well-pointed, in which $s_X$ is a closed fibrewise cofibration.
Every well-pointed fibrewise space is cofibrant but, in general,
the converse is not true.

Now we are ready for our result.

\begin{theorem}
Let $f:E\rightarrow X$ be a fibrewise pointed map between pointed
fibrant and cofibrant spaces over $B.$ Suppose, in addition, that
$B$ is a CW-complex, $X$ is a paracompact Hausdorff space and the
following conditions are satisfied:
\begin{enumerate}
\item[(i)] $f:E\rightarrow X$ is a $k$-equivalence ($k\geq 0$);

\item[(ii)] $\mbox{dim}(B)<(\mbox{secat}_B(f)+1)(k+1)-1.$
\end{enumerate}
\noindent Then $\mbox{secat}_B(f)=\mbox{secat}_B^B(f).$
\end{theorem}

\begin{proof}
Take a factorization of $f$ in $\mathbf{Top}$
$$\xymatrix{
{E} \ar[rr]^{f} \ar[dr]_{\lambda } & & {X} \\
 & {\widehat{E}} \ar[ur]_{p} & }$$ \noindent by a closed
cofibration and homotopy equivalence $\lambda $ followed by a
Hurewicz fibration $p.$ Then $\widehat{E}$ is a fibrewise pointed
space over $B$ considering $\lambda s_E$ as section and the
composite $p_Xp$ as a projection. Moreover $\lambda $ is a
fibrewise pointed homotopy equivalence as it is a homotopy
equivalence between fibrant and cofibrant spaces. Therefore, we
have the equalities
$$\mbox{secat}_B(f)=\mbox{secat}_B(p)=\mbox{secat}(p)\hspace{10pt}\mbox{and}\hspace{10pt}\mbox{secat}^B_B(f)=\mbox{secat}^B_B(p).$$

Suppose $\mbox{secat}(p)= n$ and take the iterated classical join
of $n+1$ copies of $p,$ $j_n^{p}:*^n_X\widehat{E}\rightarrow X.$
Observe that this is a Hurewicz fibration which is a fibrewise
pointed map over $B.$ Then $j^n_{p}$ admits a strict section
$\sigma :*^n_X\widehat{E}\rightarrow X$ which, necessarily, must
be a fibrewise map over $B.$ However, $\sigma $ need not be a
fibrewise \emph{pointed} map over $B.$ In order to solve this
problem we proceed as follows:

Taking into account that $j^n_{p}$ is an
$((n+1)(k+1)-1)$-equivalence over $X$ between fibrant spaces over
$X$, and that $\mbox{dim}(B)<(n+1)(k+1)-1$, we have by Lemma
\ref{fibreWhite} that
$$(j^n_{p})_*:[B,*^n_X\widehat{E}]_X\rightarrow [B,X]_X$$ is a bijection and,
in particular, is injective. Now, if
$s_{*^n_X\widehat{B}}:B\rightarrow *^n_X\widehat{B}$ denotes the
corresponding section for $*^n_X\widehat{B},$ then
$$(j^n_{p})_*([\sigma s_X])=[j^n_{p}\sigma
s_X]=[s_X]=[j^n_{p}s_{*^n_X\widehat{E}}]=(j^n_{p})_*([s_{*^n_X\widehat{E}}])$$
\noindent Therefore there exists a fibrewise homotopy over $X$,
$H:\sigma s_X\simeq _Xs_{*^n_X\widehat{B}}.$

As $s_X:B\hookrightarrow X$ is a closed cofibration and
$j^n_{p}:*^n_X\widehat{E}\rightarrow X$ a Hurewicz fibration, by
the ordinary Relative Homotopy Lifting Property we can consider a
lift in $\mathbf{Top}$
$$\xymatrix{
{X\times \{0\}\cup B\times I} \ar[rr]^(.6){(\sigma ,H)}
\ar@{^(->}[d] & & {*^n_X\widehat{E}} \ar[d]^{j^n_{p}} \\
{X\times I} \ar@{.>}[urr]^{\widetilde{H}} \ar[rr]_{pr_X} & &
{X}}$$ We define $\sigma ':=\widetilde{H}i_1:X\rightarrow
*^n_X\widehat{E}.$ Then one can straightforwardly check that
$\sigma '$ is a fibrewise pointed map over $B$ such that
$j^n_{p}\sigma '=id_X.$ From this fact one can easily check that
there exists an open cover $\{U_i\}_{i=0}^n$ of $X$ in which $U_i$
contains $s_X(B)$ and there is $s_i:U_i\rightarrow  \widehat{E} $
a strict local section of $p,$ satisfying $s_is_X=\lambda s_E,$
for all $i.$ But this implies that $\mbox{secat}^B_B(p)\leq n.$
\end{proof}

\begin{corollary}
Let $B$ be a CW-complex over $X.$ Suppose, in addition, that $X$
is a paracompact Hausdorff pointed fibrant and cofibrant space
over $B$ satisfying the following conditions:
\begin{enumerate}
\item[(i)] $s_X:B\rightarrow X$ is a $k$-equivalence ($k\geq 0$);

\item[(ii)] $\mbox{dim}(B)<(\mbox{cat}_B^*(X)+1)(k+1)-1.$
\end{enumerate}
\noindent Then the equality $\mbox{cat}_B^*(X)=\mbox{cat}_B^B(X)$
holds.
\end{corollary}

Specialising to topological complexity we obtain the following
known result \cite{Dr}.
\begin{corollary}{\rm \cite{Dr}}
Let $X$ be a $k$-connected CW-complex such that
$\mbox{dim}(X)<(\mbox{TC}(X)+1)(k+1)-1.$ Then
$\mbox{TC}(X)=\mbox{TC}^M(X).$
\end{corollary}

\end{document}